\documentclass[a4paper, 10pt, reqno]{amsproc}
\usepackage{amsmath, amssymb, amsfonts, amsthm, bm}
\usepackage[pdftex]{hyperref}
\hypersetup{
    unicode=false,          
    pdftoolbar=true,        
    pdfmenubar=true,        
    pdffitwindow=false,     
    pdfstartview={FitH},    
    pdftitle={Embedding a Latin square with transversal into a projective space},    
    pdfauthor={Lou M. Pretorius and Konrad J. Swanepoel},     
    pdfnewwindow=false,      
    colorlinks=true,       
    linkcolor=black,          
    citecolor=black,        
    filecolor=black,      
    urlcolor=black           
}

\newcommand{\setbuilder}[2]{\left\{#1\;\middle\vert\;#2\right\}}
\newcommand{\groupbuilder}[2]{\left\langle#1\;\middle\vert\;#2\right\rangle}
\DeclareMathOperator{\TD}{TD}
\newcommand{\fhi}{\varphi}
\newcommand{\numbersystem}[1]{\mathbb{#1}}
\newcommand{\D}{\numbersystem{D}}

\newcommand{\Cx}{\numbersystem{C}}

\newcommand{\Q}{\numbersystem{Q}}
\newcommand{\Z}{\numbersystem{Z}}
\newcommand{\F}{\numbersystem{F}}
\newcommand{\HQ}{\numbersystem{H}}
\newcommand{\collection}[1]{{\mathcal#1}}

\newcommand{\setT}{\collection{T}}
\newcommand{\setP}{\collection{P}}

\newcommand{\setB}{\collection{B}}

\newcommand{\card}[1]{\lvert#1\rvert}
\newcommand{\vect}[1]{\bm{#1}}
\newcommand{\vx}{\vect{x}}
\newcommand{\vv}{\vect{v}}
\newcommand{\vc}{\vect{c}}
\newcommand{\vy}{\vect{y}}
\newcommand{\vo}{\vect{o}}
\newcommand{\valpha}{\vect{\alpha}}
\newcommand{\vbeta}{\vect{\beta}}
\newcommand{\vgamma}{\vect{\gamma}}

\theoremstyle{plain}
\newtheorem{theorem}{Theorem}
\newtheorem{lemma}[theorem]{Lemma}
\newtheorem{corollary}[theorem]{Corollary}
\newtheorem{proposition}[theorem]{Proposition}

\begin{document}

\title[Embedding a Latin square]{Embedding a Latin square with
  transversal into a projective space}
\thanks{This material is based upon work supported by the 
National Research Foundation. Part of this work was done while Swanepoel was 
at the Department of Decision Science of the University of South Africa as a 
research associate.}
\author{Lou M. Pretorius}
\address{Department of Mathematics and Applied Mathematics,
        University of Pretoria,
        Pretoria 0002, South Africa}
\email{lou.pretorius@up.ac.za}
\author{Konrad J. Swanepoel}
\address{Department of Mathematics,
	London School of Economics and Political Science,
    Houghton Street,
	WC2A 2AE, London, United Kingdom}
\email{k.swanepoel@lse.ac.uk}
\date{\today}

\begin{abstract}
A Latin square of side $n$ defines in a natural way a finite geometry on $3n$ 
points, with three lines of size $n$ and $n^2$ lines of size $3$.
A Latin square of side $n$ with a transversal similarly defines a finite 
geometry on $3n+1$ points, with three lines of size $n$, $n^2-n$ lines of 
size $3$, and $n$ concurrent lines of size $4$.
A collection of $k$ mutually orthogonal Latin squares defines a geometry 
on $kn$ points, with $k$ lines of size $n$ and $n^2$ lines of size $k$.
Extending work of Bruen and Colbourn 
(J. Combin.\ Th.\ Ser.\ A \textbf{92} (2000), 88--94), 
we characterise embeddings of these
finite geometries into projective spaces over skew fields.
\end{abstract}

\maketitle

\section{Introduction}
\subsection{Definitions and notation}
A \emph{Latin square} of side $n\geq 3$ is an $n\times n$ matrix $L=[a_{ij}]$ with
entries from a set $S$ of $n$ symbols such that each symbol appears once
in each row and once in each column.
A \emph{transversal} of a Latin square $[a_{ij}]$ is a selection of
$n$ positions $(i,\sigma(i))$, $i=1,\dots, n$, no two in the same row and no two in
the same column (i.e., $\sigma$ is a permutation), such that all
symbols occur (i.e., ($a_{i,\sigma(i)})_{i=1}^n$ is a permutation of the symbols
 in $S$).
Two Latin squares $L_1=[a_{ij}]$ and $L_2=[b_{ij}]$ are \emph{orthogonal} if the $n^2$ ordered pairs $(a_{ij},b_{ij})$, $1\leq i,j\leq n$, are all distinct.
As usual, we abbreviate the term \emph{mutually orthogonal Latin squares} by MOLS.
See Section III of the Handbook of Combinatorial Designs \cite{Handbook} for a comprehensive survey on Latin squares.

A triple $(V,\setP,\setB)$ is called a
\emph{transversal design} $\TD(k,n)$ of \emph{order} $n\geq 3$ and \emph{block
size} $k\geq 3$ if $V$ is a set of size $kn$, $\setP$ a partition
of $V$
into $k$ subsets of size $n$, each called a \emph{part}, and $\setB$
is a set of
$k$-subsets of $V$, each called a \emph{block}, such that any two distinct 
elements of $V$ are contained in either a unique part or a unique block, but not
both.
This definition agrees with the definition in \cite{BrCol}, except that they use the term group instead of part.
If $X,Y\in V$ are distinct, we denote the unique part or block that contains them by $XY$.
It is well known that a Latin square of side $n$ is equivalent to a $\TD(3,n)$ by letting one part be the set of row indices, the second part the set of column indices, and the third part the set of symbols. 
The blocks are then sets of the form $\{i,j,a_{ij}\}$, where $i$ is a row index and $j$ a column index.
More generally, a collection of $k$ MOLS is equivalent to a $\TD(k+2,n)$ by duplicating the set of symbols $k$ times.
A Latin square with a transversal that has been singled out is equivalent to a $\TD(3,n)$ together with an additional partition $\setT$ of the set $V$ into $n$ pairwise disjoint blocks.

The following binary operation, associated with a $\TD(3,n)$, is fundamental to our discussion.
Let $(V,\setP,\setB)$ be a $\TD(3,n)$ with $\setP=\{P_1,P_2,P_3\}$.
Fix arbitrary points $1_1\in P_1$ and $1_2\in P_2$.
By the definition of a $\TD(3,n)$, $1_1 1_2\cap P_3$ is a singleton, say $\{1_3\}$.
We write $1_3=1_1 1_2\cap P_3$ for short.
Given any $X,Y\in P_1$, let $X'=1_2X\cap P_3$, $Y'=1_3Y\cap P_2$, and finally define $X\odot Y:=X'Y'\cap P_1$.
The equations $A\odot X=B$ and $Y\odot A=B$ both have unique solutions for all $A,B\in P_1$.
Furthermore, $1_1$ is an identity element.
Therefore, $(P_1,\odot)$ is a quasigroup with an identity, i.e., a loop \cite[III.2]{Handbook}, \cite[p.~1]{Szmielew}.

Let $\D$ be a skew field.
Denote its multiplicative group by $\D^*:=\D\setminus\{0\}$.
Let $\D^{d+1}$ denote the $(d+1)$-dimensional vector space of $(d+1)$-tuples of $\D$.
Since $\D$ is not necessarily commutative, there are two ways of multiplying a vector by a scalar.
We choose the convention that $\D^{d+1}$ is a \emph{right} vector space.
Thus for $\vx=(x_1,x_2,\dots,x_{d+1})\in\D^{d+1}$ and $\alpha\in\D$, the scalar multiple $\vx\alpha$ is defined by
\[ (x_1,x_2,\dots,x_{d+1})\alpha := (x_1\alpha,x_2\alpha,\dots,x_{d+1}\alpha).\]
We denote the zero vector by $\vo=(0,0,\dots,0)$.

Let $P^d(\D)$ be the $d$-dimensional projective space over $\D$.
We use homogeneous coordinates $[x_1,\dots,x_{d+1}]$ for a point in $P^d(\D)$, or $[x,y,z]$ when $d=2$.
Note that, since we started off with a right vector space, the homogeneous equation of a $(d-1)$-flat or \emph{hyperplane} in $P^d(\D)$ has the form 
\[\alpha_1x_1+\dots+\alpha_{d+1}x_{d+1}=0,\quad \alpha_i\in\D,\quad \text{not every $\alpha_i$ equals $0$.}\]
A \emph{pencil of hyperplanes} is a collection of all hyperplanes that contain a given $(d-2)$-flat.

An \emph{embedding of the $\TD(k,n)$ $(V,\setP,\setB)$ into $P^d(\D)$} is an injection $\fhi:V\to P^d(\D)$ such that $\fhi(P)$ is contained in a hyperplane $H_P$ of $P^d(\D)$ for each $P\in\setP$, $\fhi(B)$ is contained in a line $\ell_B$ of $P^d(\D)$ for each $B\in\setB$, and such that the hyperplanes $H_P$, $P\in\setP$, are distinct, the lines $\ell_B$, $B\in\setB$, are distinct, and no $\ell_B$ is contained in an $H_P$.
This definition of embedding coincides with the embeddings in \cite{BrCol}.
The requirement that no $\ell_B$ is contained in an $H_P$ ensures that no points in $\fhi(V)$ can lie on $H_P\cap H_Q$, where $P,Q\in \setP$ are distinct.

An \emph{embedding of a Latin square} $L$ into $P^d(\D)$ is an embedding of the associated $\TD(3,n)$.
An \emph{embedding of a Latin square with a transversal} into $P^d(\D)$ is an embedding of the associated $\TD(3,n)$ such that, if the additional partition of $V$ is $\setT=\{B_1,\dots,B_n\}$, then the lines $\ell_{B_1},\dots,\ell_{B_n}$ are concurrent.
This point of concurrency is called a \emph{transversal point} of the embedded Latin square, and will be denoted by $\infty$.
An \emph{embedding of a collection of $k$ MOLS} into $P^d(\D)$ is an embedding of the associated $\TD(k+2,n)$.
In all cases, an embedding is called \emph{proper} if $\fhi(V)$ does not lie on a hyperplane.

\subsection{Overview of the paper}
In this paper we give a full description of embeddings of Latin squares, Latin squares with transversals, and MOLS into Desarguesian projective planes and spaces, that is, projective planes and spaces over a skew field.
Motzkin \cite{MR12:849c} made a first attempt at characterising an
embedding of a Latin square into a projective plane over a field.
A correct description for this case was given by Kelly and Nwankpa \cite[theorems 3.11 and 3.12]{MR47:3207}.
Bruen and Colbourn \cite{BrCol} introduced the above notion of an embedding into $P^d(\D)$ also in the case where $\D$ is a field.
They gave a detailed description for the $2$-dimensional case, and briefly described an extension to higher dimensions \cite[Theorem 5.1]{BrCol}.
We give a complete proof of their higher-dimensional result, generalised to skew fields.
Our extension to skew fields poses only minor algebraic difficulties.
In Section~\ref{section:2d} we state without proof the $2$-dimensional cases of our results.
(They follow from the corresponding higher-dimensional results in Section~\ref{higher-dim}.)
In Section~\ref{section:groups} we discuss the finite groups that arise as subgroups of $\D^*$.
This is a much richer class of groups than the finite subgroups of fields, which are necessarily cyclic.
Section~\ref{alg} contains some algebraic preparation, and finally in Section~\ref{higher-dim} we formulate and prove all our higher-dimensional results.

\section{Embeddings into Desarguesian projective planes}\label{section:2d}
In this section we formulate the planar case of our results without proof.
Although the case where $\D$ is a field is well-known, we could not find the non-commutative versions anywhere in the literature.

In Theorem~\ref{thm1} below we show that if a Latin square with transversal is embedded in a Desarguesian projective plane, then the three parts of the corresponding $\TD(3,n)$ must lie on concurrent lines.
This generalises Case~1 of Theorem~4.1 of Bruen and Colbourn \cite{BrCol} from fields
to skew fields.
Our original motivation for such a generalisation was to show
that the $20$-point geometry obtained from the affine plane
of order $5$ by removing the $5$ points of some line,
can be embedded into $P^2(\D)$ for some skew field $\D$ only if $\D$ has characteristic $5$.
This result is used in the proof of \cite[Lemma~13]{PS}.
It is sufficient to consider the $16$-point geometry in $\F_5^2$ consisting of three parallel lines together with an additional point.
This is an embedding of a Latin square of side $5$ with transversal, and Theorem~\ref{thm1} applies.

Let $(V,\setP,\setB)$ be a $\TD(3,n)$ that is already embedded in $P^2(\D)$, i.e., $V\subseteq P^2(\D)$ and there exist three distinct lines $h_1$, $h_2$, $h_3$ of $P^2(\D)$ such that $\setP=\{V\cap h_1, V\cap h_2, V\cap h_3\}$.
We refer to this situation by saying that $(V,\setP,\setB)$ \emph{lies on the lines} $h_1$, $h_2$, $h_3$.
We now distinguish between whether $h_1$, $h_2$, $h_3$ are concurrent or not.

If the $h_i$ are concurrent, then after choosing $1_1\in h_1\cap V$ and $1_2\in h_2\cap V$, we may choose homogeneous coordinates such that the point of concurrency of the $h_i$ is $[1,0,0]$, $1_1=[0,0,1]$, $1_2=[0,1,1]$ and $1_3=[0,1,0]$.
Then the equation of $h_1$ is $y=0$, of $h_2$ is $y=z$ and of $h_3$ is $z=0$.
The coordinates of the points in $V$ depend on the choices made above.
In the next proposition we describe all possible coordinatisations.
It extends Proposition~10 of our previous paper \cite{PS}.
A further extension is found in Proposition~\ref{prop3}.
\begin{proposition}\label{prop1}
Let $(V,\setP,\setB)$ be a $\TD(3,n)$ which lies on the concurrent lines $h_1$, $ h_2$, $h_3$ of $P^2(\D)$.
If we choose homogeneous coordinates as above, then there exists a subgroup $G$ of $(\D,+)$ of order $n$ such that
\begin{equation}\label{coord1}
\left.
\begin{aligned}
  h_1\cap V&=\setbuilder{[\gamma,0,1]}{\gamma\in G},\\
  h_2\cap V&=\setbuilder{[\gamma,1,1]}{\gamma\in G},\\
  h_3\cap V&=\setbuilder{[-\gamma,1,0]}{\gamma\in G}.
\end{aligned}
\qquad\right\}
\end{equation}
The group $G$ depends only on the choice of coordinates.
For any two such choices, the two groups $G_1$ and $G_2$ so obtained satisfy $G_1=bG_2a$ for some $a,b\in\D^*$.

Conversely, given any subgroup $G$ of $(\D,+)$ of order $n$, \eqref{coord1} gives an embedding of a $\TD(3,n)$ on the concurrent lines $h_1$, $h_2$, $h_3$ with equations $y=0$, $y=z$, $z=0$, respectively.
\end{proposition}

Suppose that a skew field $\D$ contains a finite additive subgroup $G$.
Then $\D$ necessarily has prime
characteristic $p$, and $G$ is isomorphic to the direct sum of finitely many copies of $\Z_p$, the additive group of the field $\F_p$ with $p$ elements.
The next corollary is generalised in Corollary~\ref{cor3}.
\begin{corollary}\label{cor1}
Suppose that a $\TD(3,n)$ lies on three lines in $P^2(\D)$.
\begin{itemize}
 \item If $\D$ has characteristic $0$, the lines are nonconcurrent.
\item If $\D$ has prime characteristic $p$ and the lines are concurrent, then $n$ is
a power of $p$.
\end{itemize}
\end{corollary}

Now we consider the case where $h_1$, $h_2$, $h_3$ are nonconcurrent.
After choosing $1_1\in h_1\cap V$ and $1_2\in h_2\cap V$, we may choose 
homogeneous coordinates such that $1_1=[0,1,1]$, and $1_2=[1,0,1]$, 
$1_3=[1,-1,0]$, and such that $h_1$ has equation $x=0$, $h_2$ equation $y=0$, 
and $h_3$ equation $z=0$.
Again, the coordinates of the points in $V$ depend on the choices made above.
The next proposition extends Proposition~12 in the paper \cite{PS}.
A further extension is found in Proposition~\ref{prop4}.
\begin{proposition}\label{prop2}
Let $(V,\setP,\setB)$ be a $\TD(3,n)$ which lies on the nonconcurrent 
lines $h_1$, $h_2$, $h_3$ of $P^2(\D)$.
If we choose homogeneous coordinates as above, then there exists 
a subgroup $G$ of $(\D^*,\cdot)$ of order $n$ such that
\begin{equation}\label{coord2}
\left.
\begin{aligned}
  h_1\cap V&=\setbuilder{[0,\gamma,1]}{\gamma\in G},\\
  h_2\cap V&=\setbuilder{[\gamma,0,1]}{\gamma\in G},\\
  h_3\cap V&=\setbuilder{[-1,\gamma,0]}{\gamma\in G}.
\end{aligned}
\qquad\right\}
\end{equation}
The group $G$ depends only on the choice of coordinates.
For any two such choices, the two groups $G_1$ and $G_2$ so obtained are 
conjugates, i.e., $G_1=a^{-1}G_2a$ for some $a\in\D^*$.

Conversely, given any subgroup $G$ of $(\D^*,\cdot)$ of order $n$, 
\eqref{coord2} gives an embedding of a $\TD(3,n)$ on the nonconcurrent 
lines $h_1$, $h_2$, $h_3$ with equations $x=0$, $y=0$, $z=0$, respectively.
\end{proposition}

The next corollary, although purely geometric, needs some algebra in its 
proof (as can be seen in the proof of its higher-dimensional counterpart 
Corollary~\ref{cor4}).
\begin{corollary}\label{cor2}
If a $\TD(3,n)$ can be embedded in three concurrent lines of $P^2(\D)$, then 
it cannot be embedded in three nonconcurrent lines of $P^2(\D)$.
\end{corollary}

If $G$ is a subgroup of $(\D,+)$ and $a\in\D$, then $Ga$ is also a subgroup of 
$(\D,+)$, and \[\D_G:=\setbuilder{a\in\D}{Ga\subseteq G}\] is a subring of $\D$.
If $G$ is nontrivial, for any $g\in G\setminus\{0\}$, $\D_G$ is a subset of 
$g^{-1}G$, which is isomorphic to $G$.
Consequently, if $G$ is finite, $\D$ has prime characteristic $p$, say, and 
then $G$ is a $p$-group, and $(\D_G,+)$ is also a $p$-group.
(When $G$ is finite, $\D_G$ is in fact a subfield of $\D$.)

\begin{theorem}\label{thm1}
If a Latin square of side $n\geq 3$ with transversal is embedded as a $\TD(3,n)$ in
three lines $h_i$ of $P^2(\D)$ with transversal point $\infty$, then the $h_i$ are concurrent.
If homogeneous coordinates are chosen as in Proposition~\ref{prop1}, then 
the transversal point $\infty=[\gamma,a,1]$, where $\gamma\in G$, 
$a\in\D_G\setminus\{0,1\}$, and $G$ is the subgroup of $(\D,+)$ associated 
to the embedding.
Conversely, any point with these coordinates is a transversal point.

In particular, a transversal point lies on a line with equation $y=az$ 
for some $a\in\D_G\setminus\{0,1\}$.
A Latin square embedded in three concurrent lines with associated group 
$G$ has a transversal if and only if $\card{\D_G}\geq 3$.
\end{theorem}
The above theorem is generalised in Theorem~\ref{thm4}.

The next theorem gives a description of the embedding of mutually orthogonal Latin squares.
\begin{theorem}
Let $(V,\setP,\setB)$ be a $\TD(k,n)$, $n\geq 3$, $k\geq 4$ with an embedding into $P^2(\D)$ on lines $h_1,h_2,\dots,h_k$.
Then the lines $h_1,\dots,h_k$ are concurrent.
If coordinates are chosen such that $h_1,h_2,h_3$ have coordinates as in 
Proposition~\ref{prop1}, then there exist distinct 
$a_4,\dots,a_k\in\D_G\setminus\{0,1\}$ such that
\[ h_i\cap V = \setbuilder{[\gamma,a_i,1]}{\gamma\in G},\quad i=4,\dots,k,\]
where $G$ is a subgroup of $(\D,+)$ of order $n$.
Furthermore, $n$ is a prime power $p^m$, $G$ is isomorphic to $\Z_p^m$, 
$\card{\D_G}=p^t$ for some $t\leq m$, and $k\leq\card{\D_G}+1$.

In particular, if a Latin square with transversal can be embedded into 
$P^2(\D)$, then the Latin square can be extended to a $\TD(\card{\D_G}+1,n)$ 
with an embedding that extends the original embedding.
\end{theorem}
This theorem is generalised in Theorem~\ref{thm5}.

\section{The finite multiplicative subgroups of skew fields}\label{section:groups}
It is well known that any finite multiplicative subgroup of a field is cyclic.
This is in marked contrast to (non-commutative) skew fields where a much 
greater variety of finite multiplicative groups appear.
Completing earlier work of Herstein \cite{Herstein}, the finite multiplicative 
subgroups of skew fields have been characterised by Amitsur \cite{Amitsur}.
This classification is involved (see the end of this section for a partial 
formulation) and we only give a few representative examples.

As already observed by Herstein \cite{Herstein}, if $\D$ has prime 
characteristic, any finite subgroup $G$ of $\D^*$ generates a subring which 
is a finite-dimensional vector space over the prime field of $\D$, and 
therefore a subfield of $\D$.
By Wedderburn's theorem, it follows that the subring is commutative, and it 
follows that $G$ is cyclic.
Herstein similarly proved that if $G$ is an abelian subgroup of $\D^*$ 
(with $\D$ of arbitrary characteristic) then $G$ is cyclic.

The interesting case is therefore when $\D$ is a non-commutative skew field of 
characteristic $0$ and $G$ a nonabelian subgroup of $\D^*$.
The smallest such $G$ is the quaternion group of order $8$: 
$G=\{\pm 1,\pm i,\pm j,\pm k\}$, which is a subset of the quaternions $\HQ$.
By Proposition~\ref{prop2} this gives a $\TD(3,8)$ of $24$ points in $P^2(\HQ)$.
Since $G$ is nonabelian, Proposition~\ref{prop2} again gives that 
this $\TD(3,8)$ cannot be embedded in $P^2(\F)$, where $\F$ is a field.
By Corollary~\ref{cor2} it can also not be embedded on three concurrent lines 
of a projective plane over any division ring.

Coxeter \cite{Coxeter} classified the finite multiplicative subgroups of the 
quaternions $\HQ$.
Those that are not commutative, hence not conjugate to a subgroup of the 
nonzero complex numbers $\Cx^*$, are conjugate to one of the following:
\begin{enumerate}
\item The \emph{binary dihedral group}
\[D_n^*=\setbuilder{e^{ik\pi/n}, e^{ik\pi/n}j}{0\leq k<2n}\]
of order $4n$ for any $n\geq 2$ 
(with the quaternion group being the case $n=2$) giving a $\TD(3,4n)$,
\item the \emph{binary tetrahedral group} consisting of the $24$ units of 
the Hurwitz integers 
\[T^*=\left\{\pm 1,\, \pm i,\, \pm j,\, \pm k,\, 
\frac{1}{2}(\pm 1\pm i\pm j\pm k)\right\}\]
giving a $\TD(3,24)$,
\item the \emph{binary octahedral group} $O^*$ of order $48$:
\[ O^*:= T^*\cup\setbuilder{\frac{1}{\sqrt{2}}(\pm a\pm b)}{a,b\in\{1,i,j,k\},a\neq b}\]
giving a $\TD(3,48)$,
\item and the \emph{binary icosahedral group} $I^*$ of order $120$:
\[ I^*:= T^*\cup\setbuilder{\frac{1}{2}(\pm\pi_2\pm\fhi^{-1}\pi_3\pm\fhi\pi_4)}%
{\parbox{4cm}{$\pi=\pi_1\pi_2\pi_3\pi_4$ is an even\\ permutation of $\{1,i,j,k\}$}},\]
where $\fhi=(1+\sqrt{5})/2$, giving a $\TD(3,120)$.
\end{enumerate}

Amitsur found another class of groups (called D-groups in \cite{Amitsur}) 
that occur as multiplicative subgroups of division rings.
They are of the form
\[ G_{m,n,r} := \groupbuilder{a,b}{a^m=b^n=1,\; bab^{-1}=a^r},\]
where $m$, $n$, $r$ satisfy a complicated collection of relations 
\cite[Theorems 4 and 5]{Amitsur}.
In particular, $r^n\equiv 1\pmod{m}$ ensures that $\card{G_{m,n,r}}=mn$.
The smallest nonabelian multiplicative subgroup of odd order turns out to 
be $G_{7,9,2}$ of order $63$.
As demonstrated by Lam \cite{Lam2}, this group occurs in the following 
skew field.
Let $\zeta$ be a primitive $21$\textsuperscript{st} root of unity.
Introduce a new symbol $b$ that satisfies $b^3=\zeta^7$ and 
$b\zeta=\zeta^{16}b$.
Then the $\Q$-algebra 
\[\D=\setbuilder{\alpha +\beta b + \gamma b^2}{\alpha,\beta,\gamma\in\Q(\zeta)}\] 
turns out to be a division algebra, so that it is in particular, a skew field.
Note that since $[\Q(\zeta)\colon\Q]=\fhi(21)=12$, the dimension of 
$\D$ over $\Q$ is $36$.
If we set $a=\zeta^3$, then the subgroup of $\D^*$ generated by $a$ and 
$b$ is $G_{7,9,2}$.
See Lam \cite{Lam2} for further details, as well as the next largest 
example $G_{13,9,9}$ of a nonabelian multiplicative subgroup, which is of 
order $117$.

Amitsur \cite[Theorem~7]{Amitsur} proved that all non-cyclic multiplicative 
subgroups of division rings must be either of the form $G_{m,n,r}$ 
where $m$, $n$, $r$ satisfy certain properties, 
or $T^*\times G_{m,n,r}$ where $m$, $n$, $r$ satisfy certain properties, 
or $O^*$ or $I^*$.

\smallskip
\section{Some elementary algebraic lemmas}\label{alg}
\begin{lemma}\label{plemma}
In a skew field of characteristic $p$, no element can have
multiplicative order $p$.
\end{lemma}
\begin{proof}
Let $x$ be an element of multiplicative order $p$, i.e.\ $x^p=1$, $x\neq 1$.
Then by the binomial theorem modulo $p$ applied to the commuting
elements $x$ and $-1$,
\[ 0=x^p-1=(x-1)^p \neq 0,\]
a contradiction.
(Note that this argument also works for $p=2$.)
\end{proof}

\begin{lemma}\label{0lemma}
Let $G$ be a finite nontrivial subgroup of $(\D^*,\cdot)$, where $\D$ is a skew field.
Then $\sum_{g\in G} g=0$.
\end{lemma}
\begin{proof}
Consider an arbitrary $g_0\in G$.
Then
\[\sum_{g\in G} g = \sum_{g\in G} g_0g=g_0\left(\sum_{g\in G}
  g\right),\]
thus
\[ (1-g_0)\sum_{g\in G}g=0.\]
Therefore, either $\sum_{g\in G}g=0$ or $G=\{1\}$.
\end{proof}

\begin{lemma}\label{nonzero}
Let $G$ be a finite subgroup of $(\D^*,\cdot)$, where $\D$ is a skew field.
Then the order of $G$, considered as an element of $\D$, is nonzero:
\[\underbrace{1+1+\dots+1}_{\text{$\card{G}$ times}}\neq 0.\]
\end{lemma}
\begin{proof}
Suppose $\card{G}1=0$ in $\D$.
Then $\D$ has prime characteristic $p$, say, and $p$ divides
$\card{G}$.
By a theorem of Cauchy \cite[Theorem~4.2]{Rotman}, $G$ has an element
of order $p$, which contradicts Lemma~\ref{plemma}.
\end{proof}

\begin{lemma}\label{technical}
Let $G$ be a finite subgroup of $(\D^*,\cdot)$, where $\D$ is a skew field.
Suppose that $G+a=Gb$ for some $a,b\in\D$.
Then either $a=0$ or $G=\{1\}$.
\end{lemma}
\begin{proof}
Suppose $G$ is nontrivial.
Then
\begin{align*}
0 &= \sum_{g\in G} g \quad\text{(Lemma~\ref{0lemma})}\\
&= \sum_{g\in G} (-a+gb) = -\card{G}a+\left(\sum_{g\in G} g\right)b\\
&= -\card{G}a \quad\text{(again Lemma~\ref{0lemma})}.
\end{align*}
By Lemma~\ref{nonzero}, $a=0$.
\end{proof}

\section{Higher dimensions}\label{higher-dim}
Before we generalise Propositions~\ref{prop1}
and \ref{prop2}, we establish the following notation.
Let $(V,\setP,\setB)$ be a $\TD(3,n)$ embedded in $P^d(\D)$.
Thus $V\subseteq P^d(\D)$ and there exist three distinct hyperplanes 
$H_1, H_2, H_3$ of $P^d(\D)$ such that 
$\setP=\{V\cap H_1, V\cap H_2, V\cap H_3\}$.
We refer to this situation by saying that $(V,\setP,\setB)$ 
\emph{lies on the hyperplanes} $H_1, H_2, H_3$.
We now distinguish between the cases where the dimension of 
$H_1\cap H_2\cap H_3$ is $d-2$ or $d-3$.

If $\dim (H_1\cap H_2\cap H_3)=d-2$, then after choosing $1_1\in H_1\cap V$ 
and $1_2\in H_2\cap V$, we may choose homogeneous coordinates such that 
$H_1\cap H_2\cap H_3=\setbuilder{[\vx,0,0]}{\vx\in\D^{d-1}}$, 
$1_1=[\vo,0,1]$, $1_2=[\vo,1,1]$ and $1_3=[\vo,1,0]$.
(Recall that $\vo$ is the $(d-1)$-dimensional zero vector.)
Then $H_1$ has the equation $x_d=0$, $H_2$ the equation $x_d=x_{d+1}$, and 
$H_3$ the equation $x_{d+1}=0$.
The coordinates of the points in $V$ depend on the choices made above.
The next proposition describes all possibile coordinatisations.
\begin{proposition}\label{prop3}
Let $(V,\setP,\setB)$ be a $\TD(3,n)$ which lies on the hyperplanes $H_1$, 
$H_2$, $H_3$ of $P^d(\D)$ such that $\dim (H_1\cap H_2\cap H_3)=d-2$.
If we choose homogeneous coordinates as above, then there exists a subgroup 
$G$ of $(\D^{d-1},+)$ of order $n$ such that
\begin{equation}\label{coord3}
\left.
\begin{aligned}
 H_1\cap V&=\setbuilder{[\vgamma,0,1]}{\vgamma\in G},\\
 H_2\cap V&=\setbuilder{[\vgamma,1,1]}{\vgamma\in G},\\
 H_3\cap V&=\setbuilder{[-\vgamma,1,0]}{\vgamma\in G}.
\end{aligned}
\qquad\right\}
\end{equation}
The group $G$ depends only on the choice of coordinates.
For any two such choices, the two groups $G_1$ and $G_2$ so obtained, 
satisfy $G_1=TG_2a$ for some $a\in\D^*$ and $T\in GL_{d-1}(\D)$.

Conversely, given any subgroup $G$ of $(\D^{d-1},+)$ of order $n$, 
\eqref{coord3} gives an embedding of a $\TD(3,n)$ on the hyperplanes 
$H_1$, $H_2$, $H_3$ with equations $x_d=0$, $x_d=x_{d+1}$, $x_{d+1}=0$, 
respectively.
\end{proposition}
\begin{proof}
We show that the operation $\odot$ defined in the introduction corresponds 
with addition in $\D^{d-1}$.
Let $G=\setbuilder{\vgamma}{[\vgamma,0,1]\in h_1\cap V}$.
Note that $\vo\in G$.
For any $\valpha,\vbeta\in G$, let $X=[\valpha,0,1]$ and $Y=[\vbeta,0,1]$.
Then a simple calculation shows that $X'=[-\valpha,1,0]$, $Y'=[\vbeta,1,1]$, 
and $X\odot Y=[\valpha+\vbeta,0,1]$.
Therefore, $\valpha+\vbeta\in G$, and $\odot$ corresponds to addition in 
$\D^{d-1}$, restricted to $G$.
Thus $(G,+)$ is a group.
Also, $H_1\cap V$ has homogeneous coordinates as stated.
We furthermore obtain that $H_2\cap V$ and $H_3\cap V$ are as stated, by 
considering the coordinates  of the points $X'$ and $Y'$.

A calculation shows that for any two choices of coordinates as above, the 
coordinate transformation between them is $[\vx,y,z]\mapsto[T\vx,ay,az]$ for 
some $a\in\D^*$ and $T\in GL_{d-1}(\D)$.
Thus $[\vgamma,0,1]$ is mapped to $[T\vgamma,0,a]=[T\vgamma a^{-1},0,1]$, 
which gives a new group $G'=TGa^{-1}$.

The proof of the converse, that \eqref{coord3} gives a $\TD(3,n)$ for any 
subgroup $G$ of $(\D^{d-1},+)$ of order $n$, is a simple calculation.
\end{proof}

If $\D^{d-1}$ contains a finite additive group $G$, then
$\D$ has prime characteristic $p$, say, and $G$ is an $\F_p$-vector
subspace of $\D^{d-1}$ (and therefore isomorphic to a direct sum of
finitely many copies of $\Z_p$).

\begin{corollary}\label{cor3}
Let a $\TD(3,n)$ lie on three hyperplanes in $P^d(\D)$.
\begin{itemize}
 \item If $\D$ has characteristic $0$, the three hyperplanes intersect in 
 a $(d-3)$-flat.
\item If $\D$ has prime characteristic $p$ and the three hyperplanes intersect 
in a $(d-2)$-flat, then $n$ is a power of $p$.
\end{itemize}
\end{corollary}

\bigskip
Now we consider the case where $\dim (H_1\cap H_2\cap H_3)=d-3$.
After choosing $1_1\in H_1\cap V$ and $1_2\in H_2\cap V$, we may choose 
homogeneous coordinates such that  $1_1=[0,1,1,\vo]$, and $1_2=[1,0,1,\vo]$, 
$1_3=[-1,1,0,\vo]$, and such that $H_1$ has equation $x_1=0$, 
$H_2$ equation $x_2=0$, and $H_3$ equation $x_3=0$.
(Here $\vo$ is the $(d-2)$-dimensional zero vector.)
Again, the coordinates of the points in $V$ depend on the choices made above.
The next proposition describes all possible coordinatisations.
As in the two-dimensional case, there is a group associated with the 
$\TD(3,n)$, but this group now has a more complicated structure.
Define an operation on the Cartesian product $\D^*\times \D^{d-2}$ by
\[ (\alpha,\vx)\cdot(\beta,\vy):=(\alpha\beta,\vx\beta+\vy).\]
Then $\D^*\ltimes \D^{d-2}=(\D^*\times \D^{d-2},\cdot)$ is a semidirect 
product of $\D^*$ with $\D^{d-2}$ \cite[p.~137]{Rotman}, and can be faithfully 
represented in $GL_{d-1}(\D)$ by mapping $(\gamma,\vx)$ to
\[ 
\begin{bmatrix}
\gamma & \vo \\ 
\vx & I_{d-2}
\end{bmatrix}
.             \]
For any $T\in GL_{d-2}(\D)$ there is an automorphism
\[\phi_{T}: (\gamma,\vx)\mapsto (\gamma, T\vx)\]
of $\D^*\ltimes\D^{d-2}$.
\begin{proposition}\label{prop4}
Let $(V,\setP,\setB)$ be a $\TD(3,n)$ which lies on the hyperplanes $H_1$, 
$H_2$, $H_3$ of $P^d(\D)$ such that $\dim(H_1\cap H_2\cap H_3)=d-3$.
If we choose homogeneous coordinates as above, then there exists a subgroup 
$G$ of $\D^*\ltimes \D^{d-2}$ of order $n$ such that
\begin{equation}\label{coord4}
\left.
\begin{aligned}
 H_1\cap V&=\setbuilder{[0,\gamma,1,\vx]}{(\gamma,\vx)\in G},\\
 H_2\cap V&=\setbuilder{[\gamma,0,1,\vx]}{(\gamma,\vx)\in G},\\
 H_3\cap V&=\setbuilder{[-1,\gamma,0,\vx]}{(\gamma,\vx)\in G}.
\end{aligned}
\qquad\right\}
\end{equation}
The group $G$ depends only on the choice of coordinates.
For any two such choices, the two groups $G_1$ and $G_2$ so obtained 
satisfy $G_1=(a,\vv)\cdot\phi_{T}G_2\cdot(a,\vv)^{-1}$ for some 
$(a,\vv)\in\D^*\ltimes\D^{d-2}$ and $T\in GL_{d-2}(\D)$.

Conversely, given any subgroup $G$ of $\D^*\ltimes \D^{d-2}$ of order $n$, 
\eqref{coord4} gives an embedding of a $\TD(3,n)$ on the hyperplanes 
$H_1$, $H_2$, $H_3$ with equations $x_1=0$, $x_2=0$, $x_3=0$, respectively.
\end{proposition}

\begin{proof}
We calculate the loop operation $\odot$.
Let \[G:=\setbuilder{(\gamma,\vx)\in\D^*\times\D^{d-2}}%
{[0,\gamma,1,\vx]\in V\text{ for some }\vx\in\D^{d-2}}.\] 
Choose $X=[0,\alpha,1,\vx], Y=[0,\beta,1,\vy]\in H_1\cap V$.
Then easy calculations show that $X'=[-1,\alpha,0,\vx]$,
$Y'=[\beta,0,1,\vy]$, and $X\odot Y=[0,\alpha\beta,1,\vx\beta+\vy]$.
This shows that $G$ is a subgroup of $\D^*\ltimes\D^{d-2}$, 
and that the coordinates
of the $H_i\cap V$ are as stated.

A calculation shows that for any two choices of coordinates as above, 
the coordinate transformation between them is
\[ [\alpha,\beta,\gamma,\vx] \mapsto 
[a\alpha,a\beta,a\gamma,\vv(\alpha+\beta-\gamma)+T\vx] \]
for some $a\in\D^*$, $\vv\in\D^{d-2}$ and $T\in GL_{d-2}(\D)$.
Then $[0,\gamma,1,\vx]$ is mapped to
\begin{align*}
\quad & [0,a\gamma,a,\vv\gamma-\vv+T\vx]\\
&= [0,a\gamma a^{-1},1,(\vv\gamma-\vv+T\vx)a^{-1}]\\
&= [0,\beta,1,\vy]
\end{align*}
where $(\beta,\vy)=(a,\vv)\cdot(\gamma,T\vx)\cdot(a,\vv)^{-1}$,
which gives a new group $G'=(a,\vv)\cdot\phi_T G\cdot(a,\vv)^{-1}$.

The proof of the converse, that \eqref{coord4} gives a $\TD(3,n)$ for any 
subgroup $G$ of $\D^*\ltimes \D^{d-2}$ of order $n$, is again a simple 
calculation.
\end{proof}

\begin{proposition}\label{th2}
Consider an embedding of a Latin square of side $n\geq 3$ with transversal 
in $P^d(\D)$ with
transversal point $\infty$, such that the three hyperplanes of the
embedding intersect in a $(d-3)$-flat.
Then the embedding lies in a hyperplane passing through
$\infty$.
In particular, if $d=2$, such an embedding does not exist.
\end{proposition}

\begin{proof}
Suppose that $(V,\setP,\setB)$ lies on three hyperplanes $H_1, H_2, H_3$ 
that intersect in a $(d-3)$-flat.
Let $G$ be the group given by Proposition~\ref{prop4}.
The subgroup
\[G_1=\setbuilder{\gamma\in\D^*}%
{[0,\gamma,1,\vx]\in H_1\cap V \text{ for some }\vx\in\D^{d-2}}\]
of $D^*$ is a homomorphic image of the $p$-group $G$, and is therefore also
a $p$-group.
By Lemma~\ref{plemma}, $G_1$ is trivial.
Since the transversal point does not lie on any $H_i$, we may write its 
homogeneous coordinates as $[1,a,b,\vc]$ for some $a,b\in\D^*$ and 
$\vc\in\D^{d-2}$.
For any $\gamma\in G_1$ and $X=[0,\gamma,1,\vx]\in H_1\cap V$, a calculation 
then shows that the projection of $X$ from the transversal point 
$[1,a,b,\vc]$ onto $H_3$ is $[-1,\gamma b-a,0,\vx b-\vc]$.
This gives $G_1 b-a\subseteq G_1$.
Since any point in $H_3\cap V$ is such a projection of some point in 
$H_1\cap V$, we in fact have equality: $G_1 b=G_1+a$.
By Lemma~\ref{technical}, 
$G_1=\{1\}$ and $b-a=1$.
The coordinates given by Proposition~\ref{prop4} become
\begin{align*}
 H_1\cap V&=\setbuilder{[0,1,1,\vx]}{\vx\in H},\\
 H_2\cap V&=\setbuilder{[1,0,1,\vx]}{\vx\in H},\\
 H_3\cap V&=\setbuilder{[-1,1,0,\vx]}{\vx\in H},
\end{align*}
and it follows that $V$ lies on the hyperplane $x_1+x_2-x_3=0$.
\end{proof}

Similar to the two-dimensional case, if $G$ is any finite subgroup of 
$(\D^{d-1},+)$ and $a\in\D$, then $Ga$ is also a subgroup of $(\D^{d-1},+)$, 
and $\D_G:=\setbuilder{a\in\D}{Ga\subseteq G}$ is a subfield of $\D$.
As before, $(\D_G,+)$ is isomorphic to a subgroup of the $p$-group $G$, hence 
is a finite $p$-group itself.
This can be seen as follows.
Choose a coordinate $i\in\{1,\dots,d-1\}$ such that the projection $G_i$ of 
$G$ onto this coordinate is a nontrivial subgroup of $(\D,+)$.
Then $\D_{G}\subseteq\D_{G_i}\subseteq g^{-1}G$ for any $g\in G\setminus\{0\}$.
\begin{theorem}\label{thm4}
Let $(V,\setP,\setB)$ be a $\TD(3,n)$ with transversal point $\infty$ with a 
proper embedding into
three hyperplanes $H_1, H_2, H_3$ of $P^d(\D)$.
Then $\dim(H_1\cap H_2\cap H_3)=d-2$, and if
homogeneous coordinates are chosen as in Proposition~\ref{prop3} with $G$ the 
group associated with the embedding, then the transversal point 
$\infty=[\vgamma,a,1]$, where $\vgamma\in G$ and $a\in\D_G\setminus\{0,1\}$.
Conversely, any point with these coordinates is a transversal point.

In particular, a transversal point lies on a line with equation $x_d=ax_{d+1}$ 
for some $a\in\D_G\setminus\{0,1\}$.
A Latin square embedded in three hyperplanes that intersect in a $(d-2)$-flat 
has a transversal if and only if $\card{\D_G}\geq 3$.
\end{theorem}
\begin{proof}
By Proposition~\ref{th2}, $\dim(H_1\cap H_2\cap H_3)=d-2$.
Consider the group $G$ and coordinates as in Proposition~\ref{prop3}.
Since the transversal point $\infty\notin H_3$, we may write its coordinates 
as $[\valpha,\beta,1]$.
The line through $\infty$ and an arbitrary point $[-\vgamma,1,0]\in H_3\cap V$ 
intersects $H_1$ in $[\vgamma\beta+\valpha,0,1]$.
Since this point is in $V$, it has to be of the form $[\vgamma',0,1]$, and 
therefore, $\setbuilder{\vgamma\beta+\valpha}{\vgamma\in G}=G$, i.e., 
$G\beta+\valpha=G$.
It follows that $\valpha\in G$ and $\beta\in\D_G$.
Since $[\valpha,\beta,1]\notin H_1, H_2$, it follows that $\beta\neq 0,1$.

It is easily checked that for any $\valpha\in G$ and 
$\beta\in\D_G\setminus\{0,1\}$, the lines through $[\valpha,\beta,1]$ define a 
transversal of the Latin square. 
\end{proof}

\begin{corollary}\label{cor4}
Suppose that $(V,\setP,\setB)$ is a $\TD(3,n)$ that lies on three hyperplanes 
of $P^d(\D)$ that intersect in a $(d-3)$-flat.
If $(V,\setP,\setB)$ is also embeddable in three hyperplanes of $P^d(\D)$ that 
intersect in a $(d-2)$-flat, then the embedding of $V$ is not proper.
\end{corollary}
\begin{proof}
By Corollary~\ref{cor3}, if the $\TD(3,n)$ is embeddable in three hyperplanes 
that intersect in a $(d-2)$-flat, then $\D$ has prime characteristic $p$, and 
$n=p^k$ for some $k\geq 1$ and $G$ is a $p$-group.
If it furthermore lies on three hyperplanes that intersect in a $(d-3)$-flat, 
then consider the subgroup $G$ of $\D^*\ltimes\D^{d-2}$ given by 
Proposition~\ref{prop4}.
Define $G_1$ as in the proof of Proposition~\ref{th2}.
As before, $G_1$ is trivial.
As in the proof of Proposition~\ref{th2} it follows that $V$ 
(as well as the transversal point) is contained in the hyperplane
$x_1+x_2-x_3=0$.
\end{proof}
The next theorem is a generalisation of Bruen and Colbourn's Theorem 5.1 
\cite{BrCol}.
(Note that in their Theorem~5.1 it should be assumed that the embeddings are 
proper, as is already clear from Proposition~\ref{th2}.)
\begin{theorem}\label{thm5}
Let $(V,\setP,\setB)$ be a $\TD(k,n)$, $n\geq 3$, $k\geq 4$ with a proper 
embedding into $P^d(\D)$ on hyperplanes $H_1,H_2,\dots,H_k$.
Then $\dim(H_1\cap H_2\cap\dots\cap H_k)=d-2$, coordinates can be chosen such 
that $H_1,H_2,H_3$ are as in Proposition~\ref{prop3}, and there exist distinct 
$a_4,\dots,a_k\in\D_G\setminus\{0,1\}$ such that
\[ H_i\cap V = \setbuilder{[\vgamma,a_i,1]}{\vgamma\in G},\quad i=4,\dots,k,\]
where $G$ is a subgroup of $(\D^{d-1},+)$ of order $n$.
Furthermore, $n$ is a prime power $p^m$, $G$ is isomorphic to $\Z_p^m$, 
$\card{\D_G}=p^t$ for some $t\leq m$, and $k\leq\card{\D_G}+1$.

In particular, if a Latin square with transversal can be properly embedded 
into $P^d(\D)$, then the Latin square can be extended to a 
$\TD(\card{\D_G}+1,n)$ with an embedding that extends the original embedding.
\end{theorem}
\begin{proof}
We claim that each $H_i\cap V$ spans $H_i$.
Suppose to the contrary that $H_1\cap V$, say, spans a $(d-2)$-flat $F$.
Choose an arbitrary point $\infty\in V\setminus H_1$.
Without loss of generality, $\infty\in H_k$.
Then, since $\infty$ is a transversal point of the $\TD(k-1,n)$ lying on 
$H_1,\dots,H_{k-1}$, it follows that $V\cap(H_1\cup\dots\cup H_{k-1})$ lies 
on the hyperplane spanned by $F$ and $\infty$.
Similarly, $V\cap(H_1\cup H_3\cup H_4\cup\dots\cup H_{k})$ lies on the same 
hyperplane.
It follows that the whole of $V$ lies on a hyperplane, contrary to assumption.

By Theorem~\ref{thm4}, by taking any transversal point in $V$ not lying on 
three hyperplanes $H_i$, the intersection of any three hyperplanes is 
$(d-2)$-dimensional.
It follows that $\dim(H_1\cap\dots\cap H_k)=d-2$.

We may now choose coordinates such that $H_1, H_2, H_3$ are as in 
Proposition~\ref{prop3}.
By Theorem~\ref{thm4}, each point in $V\setminus(H_1\cup H_2\cup H_3)$ has 
coordinates $[\vgamma,a,1]$ with $\vgamma\in G$ and $a\in\D_G\setminus\{0,1\}$.
It remains to show for each $i=4,\dots,k$, that if 
$[\vgamma,a,1], [\vgamma',a',1]\in V\cap H_i$, then $a=a'$.
Since the $(d-2)$-flat $F=H_1\cap H_2\cap H_3\cap H_4$ also lies on the 
hyperplanes
$x_d=a x_{d+1}$ and $x_d=a' x_{d+1}$, $H_4$ is spanned by $F$ and $a$ and 
also by $F$ and $a'$.
This implies $a=a'$.
\end{proof}

\section*{Acknowledgment}
We thank the referees for their suggestions leading to an improved paper.

\end{document}